\documentclass[12pt,reqno,a4paper]{amsart}
\usepackage{amsmath,amssymb,amsthm,amsaddr}
\usepackage{enumitem}
\usepackage[margin=2.5cm,top=3.5cm,bottom=2.9cm,footskip=1cm,headsep=1cm]{geometry}
\usepackage[utf8]{inputenc}
\usepackage{graphicx}
\usepackage{tikz}
\usepackage{mathtools}
\usetikzlibrary{arrows}

\author{H. Egger \and T. Kugler}
\address{Department of Mathematics, TU Darmstadt, Germany}
\email{egger@mathematik.tu-darmstadt.de}
\email{kugler@mathematik.tu-darmstadt.de}

\title[Damped wave systems on networks]{Damped wave systems on networks:\\Exponential stability and uniform approximations}

\newtheorem{lemma}{Lemma}[section]
\newtheorem{problem}[lemma]{Problem}
\newtheorem{theorem}[lemma]{Theorem}

\theoremstyle{definition}
\newtheorem{remark}[lemma]{Remark}
\newtheorem*{example*}{Example}

\def\div{\mathrm{div}}

\def\dt{\partial_t}
\def\dx{\partial_x}
\def\dtt{\partial_{tt}}

\def\u{u}

\def\RR{\mathbb{R}}

\def\eps{\varepsilon}

\def\E{\mathcal{E}}
\def\l{l}
\def\G{\mathcal{G}}

\def\V{\mathcal{V}}
\def\Vi{{\V_{0}}}
\def\Vb{{\V_{\partial}}}

\numberwithin{equation}{section}
\numberwithin{table}{section}
\numberwithin{figure}{section}

\begin{document}

\begin{abstract} 
We consider a damped linear hyperbolic system modelling
the propagation of pressure waves in a network of pipes.
Well-posedness is established via semi-group theory
and the existence of a unique steady state is proven in the absence of driving forces.
Under mild assumptions on the network topology and the model parameters, 
we show exponential stability and convergence to equilibrium. 
This generalizes related results for single pipes and multi-dimensional domains to the network context. 
Our proof of the exponential stability estimate is based on a variational formulation of the problem, 
some graph theoretic results, and appropriate energy estimates.
The main arguments are rather generic and can be applied also for the analysis of Galerkin approximations.
Uniform exponential stability can be guaranteed for the resulting semi-discretizations under mild compatibility 
conditions on the approximation spaces. 
A particular realization by mixed finite elements is discussed
and the theoretical results are illustrated by numerical tests in which also bounds for the decay 
rate are investigated. 
\end{abstract}

\maketitle

\vspace*{-1em}

\begin{quote}
\end{quote}

\begin{quote}
\noindent
{\small {\bf AMS-classification (2000):}
35L05, 35L50, 65L20, 65M60}
\end{quote}

\section{Introduction} \label{sec:intro}

We consider the propagation of pressure waves in a network of pipes. 
On every single pipe $e$, the dynamics shall be described by the 
linear damped hyperbolic system
\begin{align*}
b^e \dt p^e + \dx u^e &= 0 \\
c^e \dt u^e + \dx p^e &= - a^e u^e.
\end{align*}
Here $p^e$ and $u^e$ denote the pressure and mass flux, respectively, and $a^e$, $b^e$, $c^e$ are positive parameters that reflect the properties of the pipe, e.g. length, 
cross-section, or roughness, and the properties of the fluid, like density or speed of sound.
The two differential equations model, respectively, the conservation of mass and the balance of momentum in the pipe $e$. 
In order to retain these physical principles also across junctions $v$ in the network, 
the mass fluxes into and the sum of forces at the junction have to balance appropriately. 
This can be phrased as algebraic coupling conditions
\begin{align*}
\sum_{e \in \E(v)} n^e(v) u^e(v) = 0 \qquad  &\qquad \text{for all } v \in \Vi \qquad \text{and} \\[-2ex]
p^e(v) = p^{e'}(v) & \qquad \text{for all } e,e' \in \E(v), \ v \in \Vi.
\end{align*}
Here $\Vi$ denotes the set of junctions $v$ in the interior of the network, 
$\E(v)$ is the set of pipes meeting at $v$,
and $n^e(v)$ takes the values minus one or one, depending on whether the pipe $e$ 
starts or ends at $v$. 
At the boundary of the network, i.e. at pipe ends $v$ not meeting at a junction, 
we assume for simplicity that the pressure is zero, i.e.,
\begin{align*}
p^e(v)=0, \qquad  v \in \Vb,
\end{align*}
where $\Vb$ denotes the set of all pipe ends $v$ at the boundary.
Inhomogeneous right hand sides or more general coupling and boundary conditions can be treated similarly.

\medskip 

The above system of differential and algebraic equations describes the evolution of pressure waves in a pipe network or the vibrations of 
a network of strings. Problems of similar structure also describe networks of electric transmission lines \cite{GoettlichHertySchillen15} 
or more general of elastic multi-structures \cite{LagneseLeugeringSchmidt}. 
Related nonlinear problems arise, for instance, in the modeling of gas pipeline networks \cite{BrouwerGasserHerty11} or of electronic circuits \cite{GuentherFeldmannTerMaten05}. 
The well-posedness of the underlying evolution problems uis usually established via semi-group theory. 
We refer to \cite{DagerZuazua06,LagneseLeugeringSchmidt,MehmetiBelowNicaise,Mugnolo14} for a collection of results concerning the modelling, analysis, and control of partial differential equations on networks. 

In general, such hyperbolic systems are governed by certain physical principles, e.g., the conservation of mass 
or the balance of momentum and energy, and dissipation or damping mechanisms lead to stability of the system. 
Depending on the topology of the network, resonances may in general occur, even in the presence of damping
\cite{DagerZuazua06,LagneseLeugeringSchmidt}.
As we will show, such problematic cases can however not arise for the damped hyperbolic system considered here.

\bigskip 

In the first part of the paper, we present a detailed stability analysis of 
the problem. 
Although the damping mechanism effectively dissipates only kinetic energy, 
one can show that, in the absence of driving forces, 
also the total energy eventually decreases, i.e., 
\begin{align*}
\sum\nolimits_e \|u^e(t)\|^2_{L^2(e)} &+  \|p^e(t)\|^2_{L^2(e)}  \\
&\le C e^{-\gamma (t-s)} \sum\nolimits_e\big( \|u^e(s)\|^2_{L^2(e)} + \|p^e(s)\|^2_{L^2(e)}\big), 
\end{align*}
for some $C$ and $\gamma>0$. 
Tthe energy thus decays exponentially to zero, and 
for time independent excitation, the system approaches steady state exponentially fast. 
Such stability estimates are well-known for damped wave equations on domains in one and multiple dimensions; 
see e.g. \cite{BabinVishik83,CoxZuazua94,Lagnese83,RauchTaylor74,Zuazua05}.
The first main result of this paper is to prove the exponential stability also in the network context. 
Let us mention that similar considerations are also of interest for the control of networks \cite{DagerZuazua06,LagneseLeugeringSchmidt,Zuazua05} and for the systematic numerical approximation \cite{BanksItoWang91,ErvedozaZuazua09,Fabiano01,RinconCopetti13,TebouZuazua03}. 

Our proof of the energy decay estimate above is follows the arguments of \cite{EggerKugler15} used for a since pipte 
and is based on the following generic ingredients: 
some graph theoretic results that allow us to proof well-posedness of the corresponding stationary problem; 
a generalized Poincar\'e inequality for certain function spaces defined on the network; 
a variational characterization of solutions to the stationary and instationary problem; 
and a decay estimate for a modified energy which serves as a Lyapunov function for the evolution. 
This last step utilizes an argument proposed originally in \cite{BabinVishik83}.

\bigskip 

In the second part of the manuscript, we investigate the systematic numerical approximation of the model problem 
by Galerkin methods, extending the ideas of \cite{EggerKugler15} for a single pipe to the network context.
Under a mild compatibility condition for the approximation spaces, we can establish the well-posedness of the Galerkin 
discretization for the stationary problem as well as the exponential stability estimate for the discretization of the evolution 
problem. The same decay rate $\gamma$ as for the continuous case can be chosen, which implies that our results are 
uniform, i.e., independent of the discretization level.
For illustration, we discuss a particular method based on the approximation by mixed finite elements, 
for which we derive mesh independent stability and convergence results. 
The exponential stability can be preserved also on the fully discrete level if appropriate time stepping schemes are used \cite{EggerKugler15}.
In summary, we thus obtain a family of uniformly exponentially stable discrete approximations for the problem under investigation.

\bigskip 

The remainder of the manuscript is organized as follows: 
In Section~\ref{sec:prelim}, we introduce the relevant notation.
In Section~\ref{sec:problem}, we state the problem under investigation in more detail 
and summarize our main analytical results. 
Proofs are given in Sections~\ref{sec:stat} and \ref{sec:instat}.
Sections~\ref{sec:stath} and \ref{sec:semi} are concerned with the Galerkin approximation 
of the stationary and the instationary problem,
and in Section~\ref{sec:fem}, we present the approximation by mixed finite elements. 
This discretization is used to illustrate our theoretical results by some numerical tests in Section~\ref{sec:num}. 
We conclude with a short discussion of our results and mention some open problems that require further research.

\section{Preliminaries and notation} \label{sec:prelim}

Let us start with recalling some elementary notations from graph theory \cite{Berge,Mugnolo14} 
that will allow us to give a convenient formulation of the problem under investigation.

\subsection{Topology}

Let $\G=(\V,\E)$ be a finite directed graph with set of vertices denoted by $\V=\{v_1,\ldots,v_n\}$ and set of edges $\E=\{e_1,\ldots,e_m\} \subset \V \times \V$. For obvious reasons we always assume that $\G$ is connected.  
To every vertex $v \in \V$ we associate a set of edges $\E(v)=\{ e=(v,\cdot) \text{ or } e=(\cdot,v)\}$ incident on $v$. 
We further denote by $\Vi=\{v : |\E(v)| \ge 2\}$ and $\Vb = \V \setminus \Vi$ the set of inner and boundary vertices. 
For every edge $e \in \E$, we define an incidence vector $(n^e)_{v \in \V}$ by 
\begin{align*}
n^e(v)=-1 \text{ if } e=(v,\cdot), \qquad 
n^e(v)=1 \text{ if } e=(\cdot,v), \qquad \text{and} \qquad 
n^e(v) = 0 \text{ else}.
\end{align*}
The role of $n^e$ is that of a normal vector for multi-dimensional problems.
The matrix $N \in \RR^{n \times m}$ defined by $N_{ij} = n^{e_j}(v_i)$ is the incidence matrix of the graph. 
For illustration of the above notions, consider the simple example given in Figure~\ref{fig:graph}
\begin{figure}[ht!]
\begin{minipage}[c]{.3\textwidth}
\hspace*{-2.5em}
\begin{tikzpicture}[scale=.7]
\node[circle,draw,inner sep=2pt] (v1) at (0,2) {$v_1$};
\node[circle,draw,inner sep=2pt] (v2) at (4,2) {$v_2$};
\node[circle,draw,inner sep=2pt] (v3) at (8,4) {$v_3$};
\node[circle,draw,inner sep=2pt] (v4) at (8,0) {$v_4$};
\draw[->,thick,line width=1.5pt] (v1) -- node[above] {$e_1$} ++(v2);
\draw[->,thick,line width=1.5pt] (v2) -- node[above,sloped] {$e_2$} ++(v3);
\draw[->,thick,line width=1.5pt] (v2) -- node[above,sloped] {$e_3$} ++(v4);
\end{tikzpicture}
\end{minipage}
\caption{\label{fig:graph}Graph $\G=(\V,\E)$ with vertices $\V=\{v_1,v_2,v_3,v_4\}$ and edges $\E=\{e_1,e_2,e_3\}$
defined by $e_1=(v_1,v_2)$, $e_2=(v_2,v_3)$, and $e_3=(v_2,v_4)$. 
Here $\Vi=\{v_2\}$, $\Vb=\{v_1,v_3,v_4\}$, and $\E(v_2)=\{e_1,e_2,e_3\}$, 
and the non-zero entries of the incidence matrix are $n^{e_1}(v_1)=n^{e_2}(v_2)=n^{e_3}(v_2)=-1$ and $n^{e_1}(v_2)=n^{e_2}(v_3)=n^{e_3}(v_4)=1$.} 
\end{figure}
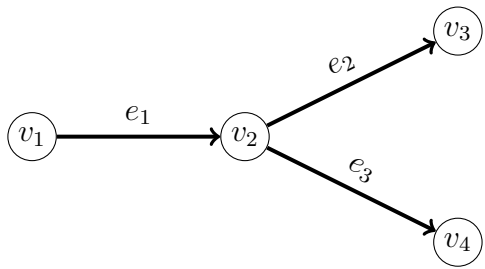

The following elementary property of graphs will be required later on, see e.g. \cite{Berge}.
\begin{lemma} \label{lem:graph}
Let $\G=(\V,\E)$ be a connected graph with incidence matrix $N \in \RR^{n \times m}$. 
Then $N$ has a regular $(n-1) \times (n-1)$ block.
\end{lemma}

\begin{remark} \label{rem:graph}
The result is proven by construction of a spanning tree.
The regular block can then be obtained by eliminating the row corresponding to the root vertex and the columns corresponding 
to the edges not present in the spanning tree. 
If there exists at least one vertex $v \in \Vb$ at the boundary, 
we can choose the root vertex of the spanning tree to lie on the boundary and eliminate it to obtain the regular subblock.. 
\end{remark}

\subsection{Geometry}

To each edge $e \in \E$, we associate a parameter $l^e>0$ representing the length of the corresponding pipe.
Throughout the presentation, we tacitly identify the interval $[0,l^e]$ with the edge $e$ which it coresponds to. 
The values $l^e$ are stored in a length vector $\l=(l^e)_{e \in \E}$.
The triple $\G=(\V,\E,\l)$ is called a \emph{geometric graph} and serves as the basic geometric model for the pipe network. 

\subsection{Function spaces}
The following function spaces defined on the geometric graph $\G=(\V,\E,\l)$
will be required for our analysis later on. 
We denote by
\begin{align*}
L^2(\E) = \{ u : u|_e=u^e \in L^2(e) \quad \forall e \in \E\}
\end{align*}
the space of square integrable functions over the network with norm
\begin{align*}
\|u\|_{L^2(\E)} = (u,u)_\E^{1/2} 
\quad \text{and} \quad 
(u,v)_\E = \sum\nolimits_e (u^e,v^e)_{L^2(e)}.
\end{align*}
For ease of presentation, we also use $\|\cdot\|_{L^2}$ and $\|\cdot\|$ to denote this norm. 
In addition to this basic function space, we will make use of broken Sobolev spaces
\begin{align*}
H^s(\E) = \{ u : u^e \in H^s(e) \quad \forall e \in \E\}. 
\end{align*}
Note that functions in $H^s(\E)$ may in general be discontinuous at interior vertices $v \in \Vi$. 
The broken derivative of a function $u \in H^1(\E)$ is denoted by $\dx' u$ defined by 
\begin{align*}
(\dx' u)|_e = \dx (u|_e) \qquad \text{for all } e \in \E. 
\end{align*}
This allows us to write $H^1(\E) = \{v \in L^2(\E) : \dx' v \in L^2(\E)\}$ with the induced norm
\begin{align*}
\|u\|_{H^1(\E)}^2 = \|u\|_{L^2(\E)}^2 + \|\dx' u\|^2_{L^2(\E)}.
\end{align*}
Similar notation will be used for functions with higher order broken derivatives.
The space $L^2(\E)$ and certain subspaces of $H^1(\E)$ will arise frequently in our analysis.

\section{Definition of the problems and main results} \label{sec:problem}

For the rest of the presentation, the pipe network will always be represented by a geometric graph $\G=(\V,\E,\l)$ 
satisfying the following conditions.
\begin{itemize}\setlength\itemsep{1ex}
 \item[(A1)] $(\V,\E)$ is a finite connected and directed graph.
 \item[(A2)] $\Vb \ne \emptyset$, i.e., there exists a least one boundary vertex.
 \item[(A3)] All pipes have finite length, i.e., $l^e>0$ for all $e \in \E$.
\end{itemize}
The phyiscal properties of the pipe and the fluid, e.g., the diameter and roughness of the pipe, 
or the density and viscosity of the fluid, are encoded in parameter functions $a,b,c$ defined on $\E$, 
which are assumed to satisfy 
\begin{itemize}\setlength\itemsep{1ex}
 \item[(A4)] $a,b,c \in L^2(\E)$ with $C_0 \le a,b,c \le C_1$ on $\E$ for some constants $C_0,C_1>0$.
\end{itemize}
We are now in the position to give a detailed formulation of the problems 
under investigation and to summarize our main analytical results, which will be stated as theorems. 

\subsection{The instationary problem}

On every edge $e$ of the network, the evolution is described by the following system of differential equations
\begin{align}
c^e \dt u^e + \dx p^e + a^e u^e &= f^e && \text{on } e \in \E, \ t>0, \label{eq:sys1}\\
b^e \dt p^e + \dx u^e &= g^e && \text{on } e \in \E, \ t>0.  \label{eq:sys2}
\end{align}
Here $f^e,g^e$ denote restrictions of appropriate functions $f,g$ defined over the network for time $t>0$ to the edge $e$.
To ensure the conservation of mass and the balance of momentum across junctions, we require the algebraic continuity and conservation conditions 
\begin{align}
p^e(v) &= p^{e'}(v) \qquad \text{for all } e,e' \in \E(v), \ v \in \Vi, \ t>0,  \label{eq:sys3}\\ 
\sum\nolimits_{e \in \E(v)}  n^e(v) u^e(v) &= 0 \qquad \qquad \text{for all } v \in \Vi, \ t>0. \label{eq:sys4}
\end{align}
At the boundary of the network, the pressure shall be prescribed by
\begin{align}
p^e(v) &= 0 \qquad  \text{for } v \in \Vb, \ e \in \E(v), \ t>0. \label{eq:sys5}
\end{align}
Inhomogeneous coupling or boundary conditions could be considered without much difficulty.
The description of the evolution is completed by the initial conditions
\begin{align}
u(0)=u_0 \qquad  p(0)=p_0 \qquad \text{on } \E \label{eq:sys6}.
\end{align}
It will be convenient for the subsequent analysis to include the 
continuity and boundary conditions \eqref{eq:sys3}--\eqref{eq:sys5} into appropriate function spaces. 
Let us therefore define
\begin{align}
H^1_0 &:= \{ p \in H^1(\E) : \eqref{eq:sys3} \text{ and } \eqref{eq:sys5} \text{ hold} \} \\
H(\div) &:= \{ u \in H^1(\E) : \eqref{eq:sys4} \text{ hold}\}.
\end{align}
These spaces are equipped with the norms inherited from $H^1(\E)$, i.e., we set
\begin{align*}
\|p\|_{H^1}^2 = \|p\|_{L^2}^2 + \|\dx' p\|_{L^2}^2 
\quad \text{and} \quad 
\|u\|_{H(\div)}^2 = \|u\|_{L^2}^2 + \|\dx' u\|_{L^2}^2.
\end{align*}
Here $\|\cdot\|_{L^2} = \|\cdot\|_{L^2(\E)}$ is the norm of $L^2(\E)$, for which we briefly write $L^2$ in the sequel.

\begin{remark}
The above notation is inspired by acoustic wave propagation in multiple space dimensions.
Note that functions $p \in H_0^1$ are \emph{continuous across junctions} $v \in \Vi$. 
The fluxes $u \in H(\div)$ may be termed \emph{conservative at junctions}, 
accordingly. 
\end{remark}
The unique solvability of the instationary problem can now be formulated as follows. 
\begin{lemma}[Well-posedness] 
Let (A1)--(A4) hold and $T>0$. 
Then for $u_0 \in H(\div)$, $p_0 \in H^1_0$, and 
$f,g \in W^{1,1}(0,T;L^2(\E))$, there exists a unique solution
\begin{align*}
(u,p) \in C^1([0,T];L^2 \times L^2) \cap C([0,T]; H(\div) \times H^1_0) 
\end{align*}
of the system \eqref{eq:sys1}--\eqref{eq:sys6} and its norm depends continuously on the norm of the data. 
Such a function $(u,p)$ is called \emph{classical solution} of the initial boundary value problem. \end{lemma}
\begin{proof}
Note that by definition of the function spaces, the coupling and boundary conditions \eqref{eq:sys3}--\eqref{eq:sys5} are satisfied automatically. 
The problem can then be understood as an abstract evolution equation on Hilbert spaces 
and the result follows by application of standard results in semi-group theory; see e.g. \cite{DL5,Evans98,Pazy83}. 
\end{proof}

\begin{remark}
Related well-posedness results for evolution equations on networks can be found for instance in \cite{Below88,Mugnolo14}. Let us note that existence could be established here also via Galerkin approximations. Detailed a-priori estimates will be derived below. 
\end{remark}

\subsection{Stationary problem}

As outlined in the introduction, we are particularly interested in the stability of the evolution and the convergence to equilibirum. Let us therefore consider next the corresponding stationary problem 
\begin{align}
\dx \bar p^e + a^e \bar u^e &= \bar f \qquad \text{on } e \in \E, \label{eq:stat1} \\
\dx \bar u^e &= \bar g \qquad \text{on } e \in \E. \label{eq:stat2}
\end{align}
The bar symbol is used here to denote functions that are independent of time. 
As before, the differential equations on the individual edges $e$ are coupled across junctions $v$ by algebraic conditions
\begin{align}
 \bar p^e(v) &= \bar p^{e'}(v) \qquad \text{for all } e,e' \in \E(v), \ v \in \Vi,  \label{eq:stat3}\\ 
\sum\nolimits_{e \in \E(v)}  n^e(v) \bar u^e(v) &= 0 \qquad \qquad \text{for all } v \in \Vi, \label{eq:stat4}
\end{align}
modelling conservation of momentum and mass across vertices $v \in \Vi$ in the interior of the network. 
At the boundary, we again require 
\begin{align}
\bar p^e(v)=0  \qquad  \text{for } v \in \Vb, \ e \in \E(v) \qquad \text{for all } v \in \Vb. \label{eq:stat5}
\end{align}
As before, the conditions \eqref{eq:stat3}--\eqref{eq:stat5} can be eliminated 
by the use of appropriate function spaces. 
Well-posedness of the stationary problem can then be stated as follows.
\begin{theorem} [Existence of a unique equilibrium] \label{thm:stat} $ $\\
Let (A1)--(A4) hold. Then for any $\bar f,\bar g \in L^2(\E)$ the stationary problem 
\eqref{eq:stat1}--\eqref{eq:stat5} has a unique solution $(\bar u,\bar p) \in H(\div) \times H_0^1$ and $\|\bar u\|_{H(\div)} + \|\bar p\|_{H^1} \le C \big( \|\bar f\|_{L^2} + \|\bar g\|_{L^2}\big)$.
\end{theorem}
\noindent 
The proof of this result will be given in Section~\ref{sec:stat}.

\subsection{Exponential stability and a-priori estimates}

From a physical point of view one would expect that the pressure waves decay in amplitude with time
in the absence of driving forces, or more generally that the system converges to equilibrium. 
This behaviour is ensured for the mathematical problem by the following stability result.

\begin{theorem}[Exponential stability] \label{thm:stability} $ $\\
Let (A1)--(A4) hold and let $(u,p)$ denote the solution of \eqref{eq:sys1}--\eqref{eq:sys5} with time independent data $f=\bar f$ and $g=\bar g\in L^2(\E)$. 
Moreover, let $(\bar u,\bar p)$ denote the solution of the 
corresponding stationary problem \eqref{eq:stat1}--\eqref{eq:stat5}. 
Then for $t \ge s \ge 0$
\begin{align} \label{eq:est1}
\| u(t) - \bar u\|^2_{L^2} + \| p(t) - \bar p\|^2_{L^2} \le C e^{-\gamma (t-s)}\big( \| u(s) - \bar u\|^2_{L^2} + \| p(s) - \bar p\|^2_{L^2}\big) 
\end{align}
with constants $C,\gamma > 0$ independent of $u$ and $p$. 
Moreover, 
\begin{align} \label{eq:est2}
\|\dt u(t)\|^2_{L^2} + \|\dt p(t)\|^2_{L^2} \le C e^{-\gamma (t-s)}\big( \|\dt  u(s) \|^2_{L^2} + \|\dt p(s)\|^2_{L^2}\big).
\end{align}
\end{theorem}
The proof of this theorem will be given given in Section~\ref{sec:instat}.
As an immediate consequence of the stability estimate, we obtain the following uniform a-priori estimates.
\begin{theorem}[Uniform a-priori estimate] \label{thm:apriori} $ $\\
Let (A1)--(A4) hold and let $(u,p)$ be a solution of \eqref{eq:sys1}--\eqref{eq:sys5}.
Then for $t \ge s \ge 0$
\begin{align}
\|u(t)\|^2 + \|p(t)) \|^2  \label{eq:est}
&\le C' e^{-\gamma (t-s)} \big( \|u(s)\|^2 + \|p(s)\|^2\big) \\
& \qquad \qquad \qquad \qquad + C'' \int_s^t e^{-\gamma (t-r)} \big( \|f(r)\|^2 + \|g(r)\|^2\big) \; dr \notag
\end{align}
with constants $\gamma,C',C''>0$ independent of $s,t$, and of the data $f,g$.
\end{theorem}
\begin{proof}
The result for the case $f=g\equiv 0$ is obtained from Theorem~\ref{thm:stability}.
The estimate for the inhomogeneous case then follows by the variation of constants formula.
\end{proof}

\begin{remark}
The stability and uniform a-priori estimates in particular imply that under assumptions (A1)--(A4), 
no resonances can occur in the pipe network. 
\end{remark}

In the following two sections, we provide the proofs for Theorem~\ref{thm:stat} and \ref{thm:stability}.
After that, we turn to the numerical approximation by Galerkin schemes, for which 
we state and prove similar results. This will form the second part of our manuscript.

\section{Analysis of the stationary problem} \label{sec:stat}

We now consider the well-posedness of the stationary problem \eqref{eq:stat1}--\eqref{eq:stat5} 
and provide a proof of Theorem~\ref{thm:stat}.
We employ a variational formulation of the problem, which later on also serves
as the starting point for the discretization by Galerkin methods.

\subsection{A variational formulation}

As a weak formulation of the stationary problem, we consider the following mixed variational problem.
\begin{problem}[Weak formulation] \label{prob:weakstat}
Find $\bar u \in H(\div)$ and $\bar p \in L^2(\E)$, such that 
\begin{align}
(a \bar u, \bar v)_\E - (\bar p, \dx' \bar v)_\E &= (\bar f, \bar v)_\E \qquad \forall \bar v \in H(\div), \label{eq:weak1stat}\\
(\dx' \bar u, \bar q)_\E &= (\bar g, \bar q)_\E \qquad \ \forall \bar q \in L^2(\E). \label{eq:weak2stat}
\end{align}
\end{problem}
Let us first clarify in detail that this problem is indeed a weak formulation of the stationary problem \eqref{eq:stat1}--\eqref{eq:stat5} under investigation. 
\begin{lemma}[Equivalence] \label{lem:equiv}
Any solution $(\bar u, \bar p) \in H^1(\E) \times H^1(\E)$ of \eqref{eq:stat1}--\eqref{eq:stat5}
also satisfies the system \eqref{eq:weak1stat}--\eqref{eq:weak2stat}. 
If, on the other hand, $(\bar u,\bar p)$ solves Problem~\ref{prob:weakstat} and is sufficiently regular, 
i.e., $(\bar u, \bar p) \in H^1(\E) \times H^1(\E)$,
then $(\bar u,\bar p)$ also solves \eqref{eq:stat1}--\eqref{eq:stat5}.
\end{lemma}
\begin{proof}
Let $(\bar u, \bar p) \in H(\div) \times H_0^1$ be a solution of \eqref{eq:stat1}--\eqref{eq:stat5}.
Then equation \eqref{eq:weak2stat} is obviously satisfied for all test functions $q \in L^2(\E)$. 
Testing \eqref{eq:sys1} with $\bar v \in H(\div)$ yields 
\begin{align*}
(\bar f, \bar v)_\E 
&= (a \bar u, \bar v)_\E -(\bar p, \dx' \bar v)_\E \\
&= (a \bar u, \bar v)_\E + (\dx' \bar p, \bar v)_\E - \sum\nolimits_{e} \bar p(v_r) \bar v(v_r) -\bar p(v_l) \bar v(v_l).
\end{align*}
The topological edge $e=(v_l,v_r)$ was tacitly identified here with its geometric representation $[0,l^e]$.
Exchanging the order of summation allows to express the last term as
\begin{align*}
\sum_{v \in \Vi} \sum_{e \in \E(v)} n^e(v) \bar v(v) \bar p(v) + \sum_{v \in \Vb} n^e(v) \bar v(v) \bar p(v).
\end{align*}
Using the algebraic conditions \eqref{eq:stat3}--\eqref{eq:stat5}, this term can be seen to vanish. 
This shows that any strong solution of \eqref{eq:stat1}--\eqref{eq:stat5} solves the variational principle. 
The other direction is obtained by reverting the order of the steps.
\end{proof}

\subsection{Auxiliary results}

Problem \eqref{eq:weak1stat}--\eqref{eq:weak2stat} has the form of an abstract mixed variational problem and well-posedness can be ensured (only) under the conditions of the Brezzi theory \cite{Brezzi74}.
For the proof of the required stability conditions, we utilize the following result, 
which follows readily from the topological properties of the network. 
\begin{lemma} \label{cor:graph}
Let (A1)--(A2) hold. 
Then for any vector $(\hat u_v)_{v \in \Vi} \in \RR^{|\Vi|}$ of \emph{nodal fluxes} there exists a vector  
$(\hat u^e)_{e \in \E} \in \RR^{|\E|}$ of \emph{constant edge fluxes} such that 
$$
\sum\nolimits_{e \in \E(v)} n^{e}(v) \hat u^e = \hat u_v \qquad \text{for all } v \in \Vi.
$$
Moreover, there holds $\max_e |\hat u^e| \le C_G \max_{v \in \Vi} |\hat u_v|$ with a constant $C_G$ depending only on the topology of the graph.
\end{lemma}
\begin{proof}
The existence of a solution follows from  Lemma~\ref{lem:graph} taking into account Remark~\ref{rem:graph}. The bound is then obtained by linearity of the problem and the finite dimension.
\end{proof}

We can now verify the conditions required for Brezzi's theorem. 
\begin{lemma}[Kernel ellipticity and inf-sup stability] \label{lem:infsup} $ $\\
Let (A1)--(A4) hold. 
Then the bilinear forms $a(u,v) = (a u, v)_\E$ and $b(u, p) = -(\dx' u, p)_\E$ 
are bounded on $H(\div) \times H(\div)$ and $H(\div) \times L^2$, respectively. 
Moreover, there exist positive constants $\alpha,\beta>0$ such that 
\begin{itemize}\setlength{\itemsep}{1ex} 
 \item[(S1)] \ $(a u, u)_\E \ge \alpha \|u\|_{H(\div)}^2$ for all $u \in H^0(\div):=\{u \in H(\div) : \dx' \u = 0 \}$;
 \item[(S2)] \ $ \sup_{u \in H(\div)} (\dx' u, p)_\E / \|u\|_{H(\div)} \ge \beta \|p\|_{L^2}$ for all $p \in L^2(\E)$.
\end{itemize}
\end{lemma}
\begin{proof}
Boundedness is clear from the definition of the norms, the Cauchy-Schwarz inequality, and the bounds for the coefficients in assumption (A4). 
The kernel ellipticity condition (S1) then  holds with $\alpha=C_0$, since 
\begin{align*}
(a u,u)_\E  \ge C_0 \|u\|_\E^2 = C_0 (\|u\|_\E^2 + \|\dx' u\|_\E^2) = C_0 \|u\|_{H(\div)}^2 \qquad \text{for all } u \in H^0(\div). 
\end{align*}
To show the inf-sup condition (S2), we proceed as follows: For every edge $e \in \E$, 
we first define $u_1^e(x) = \int_0^x  p^e(s) ds$. 
Then $u_1 \in H^1(\E)$ with $\dx' u_1 =  p$ and $\|u_1\|_\E + \|\dx' u_1\|_\E \le C \| p\|_\E$. 
The piecewise defined function $u_1$ will however not be conservative, in general. 
This can be corrected by adding a piecewise constant function $u_2$ satisfying
\begin{align*}
\sum_{e \in \E(v)} n^e(v) u_2(v) = - \sum_{e \in \E(v)} n^{e}(v) u_1(v) =: \hat u_v  \qquad \text{for all } v \in \Vi.
\end{align*}
As the following construction shows, such a function $u_2$ in fact exists: 
By Lemma~\ref{cor:graph}, we can find a vector $(\hat u^e)_{e \in \E}$ of constant edge fluxes such that 
$\sum_{e \in \E(v)} n^e(v) \hat u^e = \hat u_v$. 
We then define a piecewise constant function $u_2|_e \equiv \hat u^e$ for $e \in \E$
and the bounds of Lemma~\ref{cor:graph} yield $\|u_2^e\|_{H(\div)}=\|u_2^e\|_{L^2} \le C \|u_1\|_{H(\div)}$. 
By construction, the function $u = u_1 + u_2$ now satisfies $u \in H(\div)$ with $\dx'  u =  p$, 
and it is bounded by $\|u\|_{H(\div)} \le C \| p\|_{L^2}$. 
Using $u$ as test function in (S2) yields the assertion.
\end{proof}

\subsection{Proof of Theorem~\ref{thm:stat}}

Due to the stability estimates provided in Lemma~\ref{lem:infsup}, we can now apply Brezzi's splitting lemma \cite{BoffiBrezziFortin13,Brezzi74}, to obtain
\begin{lemma}[Well-posedness of Problem~\ref{prob:weakstat}]
Let (A1)--(A4) hold. 
Then for any pair of data $\bar f,\bar g \in L^2(\E)$, problem \eqref{eq:weak1}--\eqref{eq:weak2} 
has a unique solution $(\bar u,\bar p) \in H(\div) \times L^2$ and 
\begin{align} \label{eq:aprioristat}
\|\bar u\|_{H(\div)} + \|\bar p\|_{L^2} \le C \big(\|\bar f\|_{L^2} + \|\bar g\|_{L^2}\big) 
\end{align}
with constant $C$ only depending on $\alpha,\beta$ above and the bounds for the coefficients.
\end{lemma}

To complete the proof of Theorem~\ref{thm:stat}, it only remains to establish that 
the weak solution is sufficiently smooth and satisfies the boundary conditions, 
i.e., that $\bar p \in H^1_0$:
Testing \eqref{eq:weak1} with a smooth function supported only on a single edge $e$, we see that 
\begin{align*}
-(\bar p^e,\dx \phi^e)_e = (\bar f^e, \phi^e)_e - (a^e \bar u^e, \phi^e)_e \qquad \forall \phi_e \in C_0^\infty(e).  
\end{align*}
This shows that $\bar p$ is weakly differentiable on every edge, i.e., $\bar p \in H^1(\E)$, 
and 
\begin{align} \label{eq:pbar}
\dx' \bar p = \bar f -  a\bar u.
\end{align}
This in turn implies the bound $\|\dx' \bar p\| \le C (\|\bar f\| + \|\bar u\|)$.
Next assume that $\bar p$ is not continuous at some interior junction $v \in \Vi$.
Then $\bar p^e(v) \ne \bar p^{e'}(v)$ for some $e,e' \in \E(v)$.  
We now construct a piecewise linear function $\hat v \in H(\div)$, 
such that 
\begin{align*}
n^e(v) \bar v^e(v) + n^{e'}(v) v^{e'}(v) &= 0,  
\quad
n^{e'}(v) \hat v^{e'}(v) = 1, 
\quad \text{and} \quad 
\hat v \equiv 0 \quad \text{on } \E \setminus \{e,e'\}.
\end{align*}
By the previous considerations, we already know that $a \bar u + \dx' \bar p = \bar f$ on $\E$. 
From the variational equation \eqref{eq:weak1stat} with test function $\hat v$ as constructed above, 
we further obtain 
\begin{align*}
0 
&= (\bar f, \hat v)_\E - (a \bar u, \hat v)_\E + (\bar p, \dx'  \hat v)_\E  \\
&= (\bar f, \hat v)_\E - (a \bar u, \hat v)_\E - (\dx' \bar p, \hat v)_\E  
   + p^e(v) n^e(v) \hat v^e(v) + p^{e'}(v) n^{e'}(v) \hat v^{e'}(v).
\end{align*}
The first three terms on the right hand side vanish because of \eqref{eq:pbar},
and the remaining terms can be further rewritten as
\begin{align*}
0 = p^e(v) \big( n^e(v) \hat v^e(v) + n^{e'}(v) \hat v^{e'}(v)\big) + \big(p^e(v) - p^{e'}(v) \big) n^{e'}(v) \hat v^{e'}(v). 
\end{align*}
By construction of the test function $\hat v$, the first term vanishes, but since $\hat v^{e'}(v)=1$, the second does not,
unless $p^e(v) - p^{e'}(v)=0$. This yields a contradiction to the assumption that $\bar p$ is discontinuous at the vertex $v$;
hence $\bar p$ is continuous. 
With similar construction, one can show that 
$\bar p(v)=0$ for $v \in \Vb$, which conlucdes the proof of Theorem~\ref{thm:stat}.
\qed

\section{Analysis for the instationary problem} \label{sec:instat}

Let us now turn to the instationary problem and present the proof of Theorem~\ref{thm:stability}. 
This is accomplished by extending the arguments of \cite{EggerKugler15} to the network context. 

\subsection{Weak formulation}

As for the stationary problem, the variational characterization of the solutions turns out to be advantageous
again. Here we utilize 
\begin{problem}[Weak formulation] \label{prob:weak} 
Find a function $(u,p) \in L^2(0,T;H(\div) \cap L^2)$ with derivatives
$(c\dt u,b\dt p)\in L^2(0,T;H(div)'\times L^2)$ such that $u(0)=u_0$, $p(0)=p_0$, and
\begin{align}
(c \dt u(t),v)_\E - (p(t), \dx' v)_\E + (a u(t), v)_\E = (f(t), v)_\E \label{eq:weak1}\\
(b \dt p(t),q)_\E + (\dx' u(t), q)_\E = (g(t), q)_\E,                 \label{eq:weak2}
\end{align}
for all $v \in H(\div)$ and $q \in L^2$, and a.e. $t \in (0,T)$. 
A function $(u,p)$ satisfying these conditions is called a \emph{weak solution} of the initial boundary value problem \eqref{eq:sys1}--\eqref{eq:sys6}.
\end{problem}
As usual $H(\div)'$ denotes the dual space of $H(\div)$, 
and $(c \dt u(t),v)_\E$ is understood as duality product.
With similar arguments as for the stationary problem, we obtain
\begin{lemma}[Equivalence]  \label{lem:equiv2} 
Any classical solution $(u,p)$ of \eqref{eq:sys1}--\eqref{eq:sys6} also solves Problem~\ref{prob:weak}
and, vice versa, any weak solution $(u,p)$ of Problem~\ref{prob:weak} that is sufficiently regular is also a classical 
solution of \eqref{eq:sys1}--\eqref{eq:sys6}.
\end{lemma}

Before we proceed, let us present some auxiliary results, 
which are required later on. 
The proof of Theorem~\ref{thm:stability} will then be completed in Section~\ref{sec:proof}. 

\subsection{A generalized Poincar\'e estimate}

In the stability analysis of the stationary problem, we already encountered the space
\begin{align} \label{eq:h0div}
H^0(\div) = \{u \in H(\div) : \dx' u = 0\}
\end{align}
of piecewise constant conservative fluxes.
Note that this space is finite dimensional.
We now define a projection operator $\Pi^0 : L^2 \to H^0(\div)$, $u \mapsto u_0 := \Pi^0 u$ by
\begin{align}  \label{eq:pi0}
u^0 \in H^0(\div) : \quad (a u^0,v^0)_\E = (a u,v^0)_\E \qquad \text{for all } v^0 \in H^0(\div).
\end{align}
This finite dimensional variational problem is uniquely solvable, 
and we readily obtain
\begin{lemma}[Projection to piecewise constant fluxes] \label{lem:projection} $ $\\
Let (A1), (A4) hold. 
Then $\Pi^0 : L^2 \to H^0(\div)$ is well-defined, linear, and bounded with
\begin{align} \label{q:projection}
\|\Pi^0 u\|_{H(\div)} = \|\Pi^0u\|_{L^2}  \le C_{\Pi} \|u\|_{L^2} \qquad \text{for all } u \in L^2.
\end{align}
The stability constant can be chosen as $C_{\Pi}=\big(\frac{C_1}{C_0}\big)^{1/2}$, in particular, independent of $u$.
\end{lemma}
\begin{proof}
The operator $\Pi^0$ is the orthogonal projection with respect to the weighted 
scalar product $(a \cdot, \cdot)_\E$. 
The assertion then follows from the bounds for $a$ in assumption (A4).
\end{proof}

The following estimate plays a crucial role in our proof of the exponential stability.
 
\begin{lemma}[Generalized Poincar\'e inequality] \label{lem:poincare}  
Let (A1)--(A4) hold. 
Then 
\begin{align} \label{eq:poincare}
\|c^{1/2} u\|_{L^2}^2 \le C_P^2 \big( \|b^{-1/2} \dx' u\|_{L^2}^2 + \|a^{1/2} \Pi^0 u\|_{L^2}^2 \big) \qquad \forall u \in H(\div),
\end{align}
and the Poincar\'e constant $C_P$ can be chosen independent of $u$.
\end{lemma}

\begin{proof}
The term $\|b^{-1/2} \dx' u\|_{L^2}$ is a semi-norm on $H(\div)$ with kernel $H^0(\div)$. 
Since $H^0(\div)$ is finite dimensional, the embedding of $H(\div)$ into $H^0(\div)$ is compact. 
The last term in \eqref{eq:poincare} is also a semi-norm on $H(\div)$ and strictly positive on $H^0(\div)$. 
The assertion then follows from the lemma of equivalent norms \cite[Ch~11]{Tartar07}.  
\end{proof}

\begin{remark}
Due to the bounds for the coefficients, the right hand side of \eqref{eq:poincare} 
defines a norm which by the assertion of the Lemma is equivalent to the standard norm on $H(\div)$.
\end{remark}

The estimate \eqref{eq:equivalence} holds for general functions $u \in H(\div)$.
For solutions $(u,p)$ of Problem \ref{prob:weak}, we deduce the following bounds that will be used for our analysis later on. 
\begin{lemma}[Bounds for the $L^2$ norm] \label{lem:poincare2} $ $\\
Let (A1)--(A4) hold and $(u(t),p(t)) \in H(\div) \times L^2$ solve \eqref{eq:weak1}--\eqref{eq:weak2} with $f \equiv g\equiv0$. 
Then
 \begin{align} \label{eq:poincare2}
  \|c^{1/2} u(t)\|^2_{L^2}\le C_P^2 \big(\tfrac{C_1}{C_0} \big)\big( \|c^{1/2}\dt u(t)\|_{L^2}^2+\|b^{1/2}\dt p(t)\|_{L^2}^2 \big).
 \end{align}
\end{lemma}
\begin{proof}
We use $v=\Pi^0u(t)$ as a test function in \eqref{eq:weak1} with $f \equiv 0$. 
This yields
\begin{align*}
\|a^{1/2}\Pi^0u(t)\|^2
&= (a u(t),\Pi^0u(t))_{\E} = -(c\dt u(t),\Pi^0u(t))_{\E} \le \|c^{1/2} \dt u(t)\| \|c^{1/2} \Pi^0 u(t)\|.
\end{align*}
Together with \eqref{eq:weak2} for $g \equiv 0$ and with the bounds for the coefficients, we obtain 
\begin{align*}
\|a^{1/2}\Pi^0u(t)\|^2 \le \tfrac{C_1}{C_0} \|c^{1/2} \dt u(t)\|^2
\qquad \text{and} \qquad 
\|b^{-1/2} \dx' u(t)\|^2 \le \tfrac{C_1}{C_0} \|b^{1/2} \dt p(t)\|^2.
\end{align*}
The assertion now follows from these bounds and the Poincar\'e inequality \eqref{eq:poincare}.
\end{proof}

Theorem~\ref{thm:stability} can now be proven with similar techniques as the corresponding result for a single pipe \cite{EggerKugler15}. 
For convenience of the reader and to keep track of the constants, we recall in the following the main steps of the proof.

\subsection{Energy estimates}

We consider Problem~\ref{prob:weak} with data $f \equiv \bar f$ and $g \equiv \bar g$
independent of time and start with the second estimate of Theorem~\ref{thm:stability}.
Define the \emph{energy}
\begin{align*}
E(t) := \frac{1}{2} \big( \|c^{1/2}\dt u(t)\|_{L^2}^2+\|b^{1/2}\dt p(t)\|_{L^2}^2 \big).
\end{align*}
By differentiation of \eqref{eq:weak1}--\eqref{eq:weak2} with respect to time, we see that 
\begin{align}
(c \dtt u(t),v)_\E - (\dt p(t), \dx' v)_\E + (a \dt u(t), v)_\E = 0 \label{eq:weakt1}\\
(b \dtt p(t),q)_\E + (\dx'\dt  u(t), q)_\E = 0 \label{eq:weakt2}
\end{align}
for all $v \in H(\div)$ and $q \in L^2$ and a.e. $t>0$. 
For $v = \dt u(t)$ and $q = \dt p(t)$, we obtain
\begin{align} \label{eq:decay1}
\frac{d}{dt} E(t) = -(a \dt u(t), \dt u(t))_\E \le 0.
\end{align}
Hence $E$ is a Lyapunov functional for the evolution problem \eqref{eq:weak1}--\eqref{eq:weak2}. 
This estimate is however not sufficient to guarantee exponential decay of the energy.
Following an idea introduced first in \cite{BabinVishik83}, 
see also \cite{EggerKugler15,Zuazua88}, 
we consider additionally a \emph{modified energy}
\begin{align*}
 E_{\eps}(t) & := E(t) + \eps(c\dt u(t),u(t))_{\E}.
\end{align*}
For appropriate choice of $\eps$, the two energies can be shown to be equivalent.
\begin{lemma}[Equivalence] \label{lem:equivalence}
Let (A1)--(A4) hold and $|\eps| \le \frac{C_0}{4 C_1C_P}$. 
Then
 \begin{align} \label{eq:equivalence}
  \frac{1}{2}E(t)\le E_{\eps}(t)\le \frac{3}{2}E(t).
 \end{align}
\end{lemma}
\begin{proof}
By means of Lemma~\ref{lem:poincare2}, the additional term can be estimated by
\begin{align*}
|(c\dt u,u)_{\E}|
 & \le \|c^{1/2}\dt u\|\| c^{1/2} u\| \\
 & \le \|c^{1/2} \dt u\| C_P \big(\tfrac{C_1}{C_0}\big)^{1/2} \big( \|c^{1/2}\dt u(t)\|^2+\|b^{1/2}\dt p(t)\|^2 \big)^{1/2}  \le \tfrac{2 C_1 C_P}{C_0} E(t).
 \end{align*}
The assertion now follows by scaling with $\eps$ and some elementary calculations. 
\end{proof}

We next show that the modified energy $E_\eps$ also defines a Lyapunov functional for the evolution and, moreover, $E_\eps$ decreases exponentially along solution trajectories.
\begin{lemma}[Energy dissipation] \label{lem:decay}
Let 
$0<\eps\le\frac{C_0}{C_1} \frac{C_0}{2 C_0 + 4C_P C_1}=:\eps^*$. 
Then 
\begin{align} \label{eq:decayE}
\frac{d}{dt} E_\eps(t) \le - \frac{2\eps}{3} E_\eps(t).
\end{align}
\end{lemma}
\begin{proof}
From the definition of $E_\eps$ and \eqref{eq:decay1}, we immediately get 
\begin{align*}
\frac{d}{dt}E_{\eps}(t)
&= \frac{d}{dt} E(t) + \eps \frac{d}{dt} (c\dt u(t),u(t))_{\E}\\
&\le -\|a^{1/2} \dt u(t)\|^2 + \eps \|c^{1/2} \dt u(t)\|^2 + \eps(c\dtt u(t),u(t))_{\E}.
\end{align*}
Using the variational principles \eqref{eq:weak1}--\eqref{eq:weak2} and \eqref{eq:weakt1}--\eqref{eq:weakt2}
characterizing $(u,p)$ and $(\dt u,\dt p)$, the bounds for the coefficients, and the bound \eqref{eq:poincare2}, 
we can estimate the last term by
\begin{align*}
(c\dtt u(t),u(t))_{\E}
  &=-(a\dt u,u)_\E-(b\dt p,\dt p)_\E\\
  &\le \big(\tfrac{C_1}{C_0}\big)^{1/2} \|c^{1/2} \dt u\| \|c^{1/2} u\| - \|b^{1/2}\dt p(t)\|^2 \\
  &\le C_P \big(\tfrac{C_1}{C_0}\big) \|c^{1/2}\dt u(t)\| \big( \|c^{1/2} \dt u(t)\|^2 + \|b^{1/2} \dt p(t)\|^2 \big)^{1/2} - \|b^{1/2}\dt p(t)\|^2. 
\end{align*}
Scaling with $\eps$ and an application of Young's inequality further yield
\begin{align*}
\eps (c\dtt u(t),u(t))_{\E}
  &\le \eps \tilde C \big( \tfrac{\tilde c}{2} + \tfrac{1}{2\tilde c}\big) \|c^{1/2} \dt u(t)\|^2 - \tfrac{\eps}{2} \|b^{1/2}\dt p(t)\|^2
\end{align*}
with constant $\tilde c = \tfrac{C_1 C_P}{C_0}$. 
Together with the above expression for $\frac{d}{dt} E_\eps(t)$, this leads to
\begin{align*}
\frac{d}{dt} E_\eps(t) 
&\le -\big(\tfrac{C_0}{C_1} - \tfrac{3}{2}\eps - \eps \tfrac{\tilde c^2}{2}\big) \|c^{1/2}\dt u(t)\|^2 - \tfrac{\eps}{2} \|b^{1/2}\dt p(t)\|^2.
\end{align*}
From the bounds for the parameter $\eps$, we can thus conclude that
\begin{align*}
\frac{d}{dt} E_\eps(t) \le - \eps E(t).
\end{align*}
The assertion then follows by equivalence of the two energies $E$ and $E_\eps$.
\end{proof}

\subsection{Proof of Theorem~\ref{thm:stability}} \label{sec:proof}

We are now in the position to complete the proof of Theorem~\ref{thm:stability}. 
Let us start with the second estimate: From Lemma~\ref{lem:decay}, we obtain
\begin{align*}
E_\eps(t) \le e^{-2 \eps^* (t-s)/3} E_\eps(s) \qquad \text{for all } t \ge s. 
\end{align*}
By Lemma~\ref{lem:equivalence}, we thus obtain \eqref{eq:est2} 
with $C=3$ and $\gamma=2\eps^*/3$ and $\eps^*$ as in Lemma~\ref{lem:decay}.

\medskip 

The first estimate \eqref{eq:est1} can now be deduced from \eqref{eq:est2} with the following arguments: 
Let $(\tilde u,\tilde p) \in H(\div) \times H_0^1$ denote the weak solution of the auxiliary stationary problem
\begin{align*}
a \tilde u + \dx' \tilde p &= u_0 - \bar u, \\
\dx' \tilde u &= p_0 - \bar p. 
\end{align*}
Due to the choice of the spaces, the continuity and boundary conditions \eqref{eq:stat3}--\eqref{eq:stat5} are 
satisfied automatically. 
By elementary calculations, one can see that the functions
\begin{align*}
U(t)=\int_0^tu(s)-\bar u \; ds - \tilde u 
\quad \text{and} \quad 
P(t)=\int_0^tp(s)-\bar p \; ds - \tilde p
\end{align*}
then satisfy the variational equations \eqref{eq:weak1}--\eqref{eq:weak2} with $f \equiv g\equiv 0$. 
Applying the second estimate \eqref{eq:est2} of Theorem~\ref{thm:stability} to $(U,P)$ instead of $(u,p)$, we obtain 
\begin{align*}
\|c^{1/2} \dt U(t)\|^2 + \|b^{1/2} \dt P(t)\|^2 \le C e^{-\gamma (t-s)} \big( \|c^{1/2} \dt U(s)\|^2 + \|b^{1/2} \dt P(s)\|^2\big).
\end{align*}
Since $\dt U(t) = u(t) - \bar u $ and $\dt P(t) = p(t) - \bar p$, this already yields the estimate \eqref{eq:est1}
and concludes the proof of Theorem~\ref{thm:stability}. \hfill \qed
%

\section{Discretization of the stationary problem} \label{sec:stath}

The proof of the well-posedness for the stationary problem was based on a variational characterization of solutions. This suggests to use Galerkin schemes for discretization. 

\subsection{Galerkin approximation}

Let $V_h \subset H(\div)$ and $Q_h \subset L^2$ be finite dimensional subspaces. 
For the discretization of the stationary problem, we consider conmforming Galerkin approximations of the following form. 

\begin{problem}[Space discretization] \label{prob:stath}
Find $\bar u_h \subset V_h$ and $\bar p_h \subset Q_h$ such that  
\begin{align}
(a \bar u_h, \bar v_h)_\E - (\bar p_h, \dx' \bar v_h)_\E &= (\bar f, \bar v_h)_\E \qquad \forall \bar v_h \in V_h \label{eq:weak1stath}\\
(\dx' \bar u_h, \bar q_h)_\E &= (\bar g, \bar q_h)_\E \qquad \ \forall \bar q_h \in Q_h. \label{eq:weak2stath}
\end{align}
\end{problem}

A particular realization of such a method by a mixed finite element approximation will be discussed in some detail in Section~\ref{sec:fem} below. 

\subsection{Stability and error analysis}

In order to ensure the well-posedness of the discrete variational problem, 
we require some basic conditions for the approximation spaces. 
In the sequel, we will therefore assume that
\begin{itemize}\setlength\itemsep{1ex}
 \item[(A5)] $V_h \subset H(\div)$ and $Q_h \subset L^2$ are finite dimensional;
 \item[(A6)] $\dx' V_h = Q_h$; 
 \item[(A7)] $H^0(\div) \subset V_h$.
\end{itemize}
The compatibility conditions (A6)--(A7) in particular ensure that \eqref{eq:weak2stath} is solvable. 
The assumptions (A5)--(A7) further allow us to prove the following discrete stability conditions.

\begin{lemma} \label{lem:infsuph}
Let (A1)--(A7) hold. Then 
\begin{itemize}\setlength\itemsep{1ex}
 \item[(S1h)] \ $(a u_h, u_h)_\E \ge \alpha \|u_h\|_{H(\div)}^2$ for all $u_h \in V_h^0=\{u_h \in V_h : (\dx' \bar u_h,q_h)_\E = 0 \ \forall q_h \in Q_h\}$;
 \item[(S2h)] \ $ \sup_{u_h \in V_h} (\dx' u_h, p_h)_\E / \|u_h\|_{H(\div)} \ge \beta \|p_h\|_{L^2}$ for all $p_h \in L^2(\E)$.
\end{itemize}
The stability constants $\alpha,\beta$ can be chosen the same as those in Lemma~\ref{lem:infsup}. 
\end{lemma}
\begin{proof}
The proof of Lemma~\ref{lem:infsup} applies almost verbatim also to the discrete setting:\\
The condition $\dx' V_h \subset Q_h$ ensures that $V_h^0 \subset H^0(\div)$.
This already yields the kernel ellipticity (S1h) with the same constant as on the continuous level. 
The two conditions $\dx' V_h \supset Q_h$ and $H^0(\div) \subset V_h$ allow us to apply the proof of 
condition (S2) in Lemma~\ref{lem:infsup} almost verbatim also on the discrete level.
\end{proof}

As a direct consequence of the previous lemma and the Brezzi theory, we obtain
\begin{theorem}[Error estimates] \label{thm:eestat}
Let (A1)--(A7) hold. Then for any $\bar f,\bar g \in L^2(\E)$,
Problem~\ref{prob:stath} has a unique discrete solution $(\bar u_h,\bar p_h) \in V_h \times Q_h$. 
Moreover, 
\begin{align*}
\|\bar u - \bar u_h\|_{H(\div)} + \|\bar p - \bar p_h\|_{L^2} 
\le C \big( \inf_{\bar v_h \in V_h} \|\bar u - \bar v_h\|_{H(\div)} + \inf_{q_h \in Q_h} \|\bar p - \bar q_h\|_{L^2} \big) 
\end{align*}
with constant $C$ depending only on the $\alpha$, $\beta$, and the bounds for the coefficients.
\end{theorem}
\begin{proof}
The assertion follows from standard results about the Galerkin approximation of mixed variational problems; 
see  \cite{Brezzi74} or \cite[Ch.~5]{BoffiBrezziFortin13} for details.
\end{proof}
\begin{remark}
Let us mention that somewhat stronger estimates for the discretization error can be obtained by further employing the compatibility condition (A6); see \cite[Ch.~5]{BoffiBrezziFortin13} for details. 
Particular examples of such estimates are given in Section~\ref{sec:fem} below.
\end{remark}

\subsection{Elliptic projection}

The discrete variational problem allows us to associate to any function
$(\bar u, \bar p) \in H(\div) \times L^2$ a discrete function $(\bar u_h,\bar p_h) \in V_h \times Q_h$ via
\begin{align*}
(a \bar u_h, \bar v_h)_\E - (\bar p_h, \dx' \bar v_h)_\E 
&= (a \bar u, \bar v_h)_\E - (\bar p, \dx' \bar v_h)_\E && \forall \bar v_h \in V_h \\
(\dx' \bar u_h, \bar q_h)_\E &= (\dx' \bar u, \bar q_h)_\E  && \forall \bar q_h \in Q_h. \end{align*}
This defines the \emph{elliptic projection}
$\Pi_h : H(\div) \times L^2 \to V_h \times Q_h$, $(\bar u,\bar p) \mapsto (\bar u_h,\bar p_h)$.
The following properties directly follow from the construction and the previous results.
\begin{lemma}[Elliptic projection]
The operator $\Pi_h:H(\div) \times L^2 \to V_h \times Q_h$ defined above is linear and bounded 
and leaves $V_h \times Q_h$ invariant. Moreover, 
\begin{align*}
\|\Pi_h(\bar u,\bar p)\|_{H(\div)\times L^2} \le C \|(\bar u,\bar p)\|_{H(\div) \times L^2} \qquad \forall (\bar u,\bar p) \in H(\div) \times L^2.
\end{align*}
\end{lemma} 
The bound follows in the same way as Theorem~\ref{thm:eestat}.
Again, somewhat sharper estimates can be obtained by a refined analysis, as we will shown in Section~\ref{sec:fem} below.
%

\section{Semi-discretization of the instationary problem} \label{sec:semi}

The Galerkin approximation of the stationary problem can be extended without difficulty to the the variational formulation of the instationary problem. 

\subsection{Galerkin discretization}

Let $V_h \subset H(\div)$ and $Q_h \subset L^2$ be finite dimensional subspaces
and choose some $T>0$. 
For the discretization of the instationary problem, we consider Galerkin approximations of the following form. 

\begin{problem}[Semi-discretization] \label{prob:instath}
Find $(u_h,p_h) \in H^1(0,T;V_h \cap Q_h)$ with initial values 
$(u_h(0),v_h(0)) = \Pi_h(u_0,p_0)$ defined by the elliptic projection,
and such that 
\begin{align}
(c \dt u_h(t),v_h)_\E - (p_h(t), \dx' v_h)_\E + (a u_h(t), v_h)_\E = (f(t), v_h)_\E \label{eq:weak1h}\\
(b \dt p_h(t),q_h)_\E + (\dx' u_h(t), q_h)_\E = (g(t), q_h)_\E,                 \label{eq:weak2h}
\end{align}
for all test functions $v_h \in V_h$ and $q_h \in Q_h$, and every $t \in [0,T]$.
\end{problem}

By choice of a basis, the discrete variational problem 
can be turned into a linear system, and the existence of a unique solution 
follows by the Picard-Lindelöf theorem. 
\begin{lemma}
Let (A1)-(A5) hold, $u_0 \in H(\div)$, $p_0 \in L^2$, 
and $f,g \in L^2(0,T;L^2(\E))$. 
Then Problem~\ref{prob:instath} has a unique solution depending continuously
on the data.
\end{lemma}
\begin{remark}
The error analysis for the Galerkin approximation can now be carried out 
in the usual way; see e.g. \cite{CowsarDupontWheeler96,Geveci88}. 
Unfortunately, the constants in the error estimates will depend on the time horizon $T$, 
which prohibits an investigation of the long-term behaviour. 
To obtain estimates that are uniform in $T$, 
a more detailed stability analysis for the discrete problems is required. 
\end{remark}

\subsection{Exponential stability and uniform a-priori estimates}

Let $f \equiv \bar f$ and $g \equiv \bar g$ be independent of time. 
In this case, the solution $(u(t),p(t))$ of the instationary problem \eqref{eq:sys1}--\eqref{eq:sys5} was shown to converge to the equilibrium 
$(\bar u,\bar p)$ exponentially fast.
This behaviour is preserved by the Galerkin approximations discussed above.

\begin{theorem}[Discrete exponential stability] $ $ \label{thm:stabh}\\
Let (A1)--(A7) hold and let $(\bar u_h,\bar p_h)$ and $(u_h,p_h)$ be the solutions of Problem~\ref{prob:stath} and \ref{prob:instath} 
with $f \equiv \bar f$ and $g \equiv \bar g$ independent of time.
Then 
\begin{align*}
\|u_h(t) - \bar u_h\|_{L^2}^2 + \|p_h(t) - \bar p_h\|_{L^2}^2 \le C e^{-\gamma (t-s)} \big( \|u_h(s) - \bar u_h\|_{L^2}^2 + \|p_h(s) - \bar p_h\|_{L^2}^2 \big) .
\end{align*}
The constants $C,\gamma>0$ can be chosen the same as those in Theorem~\ref{thm:stability}.  
\end{theorem}
\begin{proof}
The proof of Theorem~\ref{thm:stability} applies almost verbatim.
For convenience of the reader, we again sketch the main steps: 
We first define discrete energies $E_h$ and $E_{\eps,h}$ and show their equivalence; 
the proof of Lemma~\ref{lem:equivalence} applies verbatim.
As a next step, we establish a discrete version of the energy dissipation estimate in
Lemma~\ref{lem:decay}; again, the proof applies verbatim also on the discrete level. 
The discrete stability estimates are then obtained with the same arguments as 
presented in Section~\ref{sec:proof}.
\end{proof}

As a direct consequence of the discrete exponential stability estimates, 
we now obtain the following uniform a-priori bounds for the Galerkin approximations.
\begin{theorem}[Discrete a-priori bounds] \label{thm:apriorih} $ $\\
Let (A1)--(A7) hold and let $(u_h,p_h)$ denote the solution of Problem~\ref{prob:instath}.
Then 
\begin{align}
\|u_h(t)\|^2 + \|p_h(t) \|^2  \label{eq:esth}
&\le C' e^{-\gamma (t-s)} \big( \|u_h(s)\|^2 + \|p_h(s)\|^2\big) \\
& \qquad \qquad \qquad \qquad + C'' \int_s^t e^{-\gamma (t-r)} \big( \|f(r)\|^2 + \|g(r)\|^2\big) \; dr \notag
\end{align}
with constants $\gamma,C',C''>0$.
The decay rate $\gamma$ is the same as in Theorem~\ref{thm:stabh}.
\end{theorem}
\begin{proof}
The proof follows with the same arguments as that of Theorem~\ref{thm:apriori}. 
\end{proof}

\subsection{Error estimates}
We can now state the basic error estimates for the Galerkin discretizations proposed above. 
We do this in order to illustrate that the estimates are uniform with respect to time, 
and again only sketch the main arguments of the proofs.

\begin{theorem} \label{thm:eeinstat}
Let (A1)--(A7) hold and let $(u,p)$ and $(u_h,p_h)$ be the solutions of Problems~\ref{prob:weak} and \ref{prob:instath}, respectively. 
Moreover, set $(\widetilde u_h(t),\widetilde p_h(t)) = \Pi_h(u(t),p(t))$.
Then
\begin{align*}
&\|u(t)-u_h(t)\|^2_{L^2} + \|p(t) - p_h(t)\|^2_{L^2} 
\le \|u(t) - \widetilde u_h(t)\|^2_{L^2} + \| p(t) - \widetilde p_h(t)\|^2_{L^2} \\
& \qquad \qquad\qquad  + C'' \int_0^t e^{-\gamma (t-s)} \big( \|\dt u(s) - \dt \widetilde u_h(t)\|^2_{L^2} + \| \dt p(s) - \widetilde p_h(t)\|^2_{L^2} \big) \; ds.
\end{align*}
The constants $\gamma,C',C''>0$ are independent of $t$ and the functions $u$ and $p$.
\end{theorem}
\begin{proof}
As suggested in \cite{Varga,Wheeler73}, we can split the error 
into 
\begin{align*}
&\|u(t) - u_h(t)\| + \|p(t) - p_h(t)\| \\ 
& \le \big( \|u(t) - \widetilde u_h(t)\| + \|p(t) - \widetilde p_h(t)\| \big) 
    + \big( \|\widetilde u_h(t) - u_h(t)\| + \|\widetilde p_h(t) - p_h(t)\| \big).
\end{align*}
The first term on the right hand side already appears in the final estimate.
To bound the second term, we set
$w_h = \widetilde u_h(t) - u_h(t)$ and $r_h = \widetilde p_h(t) - p_h(t)$,
and note that $(w_h,r_h)$ satisfies $w_h(0)=0$ and $r_h(0)=0$ and, in addition,
\begin{align*}
(c \dt w_h(t),v_h)_\E - (r_h(t), \dx' v_h)_\E + (a w_h(t), v_h)_\E = (\widetilde f(t), v_h)_\E \qquad \forall v_h \in V_h \\
(b \dt r_h(t),q_h)_\E + (\dx' w_h(t), q_h)_\E = (\widetilde g(t), q_h)_\E \qquad \forall q_h \in Q_h
\end{align*}
with right hand sides $\widetilde f(t) = \dt \widetilde u_h(t) - \dt u(t)$ and $\widetilde g(t) = \dt \widetilde p_h(t) - \dt p(t)$.
Here we used the properties of the elliptic projection. 
The assertion then follows from the stability estimate of Theorem~\ref{thm:apriorih}.
\end{proof}

Similar as for the stationary problem, sharper estimates can be obtained by using 
the compatibility condition (A6) and a refined error analysis; an example will be given  below. 
For time independent right hand sides, the error estimate simplifies substantially.

\begin{theorem}
Let the assumptions and notations of Theorem~\ref{thm:eeinstat} hold.  
Moreover, assume that $f \equiv \bar f$ and $g \equiv \bar g$, 
and let $(\bar u,\bar p)$ and $(\bar u_h,\bar p_h)$ denote, respectively, 
the solution of the stationary problem and its discrete approximation.
Then
\begin{align*} 
&\|u(t)-u_h(t)\|^2_{L^2} + \|p(t) - p_h(t)\|^2_{L^2}  \\
& \qquad \qquad \qquad \qquad \le \|\bar u -  \bar u_h\|^2_{L^2} + \| \bar p - \bar p_h\|^2_{L^2}  + C''' t e^{-\gamma t}.
\end{align*}
\end{theorem}
\begin{proof}
The result follows from the estimate of Theorem~\ref{thm:eeinstat}, 
the exponential decay  estimates of Theorem~\ref{thm:stability} and \ref{thm:stabh}, 
and the triangle inequality. 
\end{proof}
On the long run, the discretization error is therefore dominated by the approximation 
of the stationary problem, which can be expected because of convergence to equilibrium.

\section{A mixed finite element method} \label{sec:fem} 

We now give a concrete example of a stable Galerkin approximation based on discretization by finite elements. To fully explain the numerical results presented later on, we derive somewhat improved error estimates for this particular discretization.

\subsection{The mesh and polynomial spaces}

Let $[0,l^e]$ be the interval represented by the edge $e$.  
We denote by $T_h(e) = \{T\}$ a uniform mesh of $e$ with subintervals $T$ of length $h^e$. 
The global mesh is then defined as $T_h(\E) = \{T_h(e) : e \in \E\}$, 
and the global mesh size is denoted by $h=\max_e h^e$. 
We denote the spaces of piecewise polyonomials on $T_h(\E)$ by
\begin{align*}
 P_k(T_h(\E)) &= \{v \in L^2(\E) : v|_e \in P_k(T_h(e)), \ e \in \E\} 
\end{align*}
where $P_k(T_h(e)) = \{v \in L^2(e) : v|_T \in P_k(T), \ T \in T_h(e)\}$ and $P_k(T)$ 
is the space of polynomials of degree $\le k$ on the subinterval $T$. 
Note that $P_k(T_h(\E)) \subset L^2(\E)$, which is easy to see, 
but in general $P_k(T_h(\E)) \not\subset H^1(\E)$. 

\subsection{The mixed finite element approximation}

As spaces $V_h$ and $Q_h$ for the Galerkin approximation presented in the previous sections, 
we now consider
\begin{align} \label{eq:spaces}
V_h = P_{1}(T_h(\E)) \cap H(\div)
\quad \text{and} \quad 
Q_h = P_{0}(T_h(\E)). 
\end{align}
Corresponding higher order approximations could be utilized as well. 
This choice of spaces can be shown to satisfy the required compatibility conditions.
\begin{lemma}
The spaces $V_h$, $Q_h$ defined above satisfy the assumptions (A5)--(A7). 
\end{lemma}
\begin{proof}
$V_h$, $Q_h$ are finite dimensional and clearly $\dx' V_h \subset Q_h$. 
Since functions in $H^0(\div)$ are constant on each edge $e$, we also obtain   
$H^0(\div) \subset V_h$. To see that $\dx' V_h \supset Q_h$, we have to provide for any $q_h \in Q_h$ 
a function $v_h \in V_h$ with $\dx' v_h = q_h$. This can be achieved with the same 
construction as in the proof of Lemma~\ref{lem:infsup}.
\end{proof}

As a consequence, all stability results, the a-priori bounds, and error estimates of the previous sections apply to the Galerkin approximations based on these finite element spaces. 
This will be illustrated by numerical results in the next section. 
To obtain quantitative error estimates, we will make use of the following interpolation error results. 
\begin{lemma}[Approximation] \label{lem:approx}
Let $V_h$, $Q_h$ be chosen as above. 
Then there exist generalized interpolation operators $\Pi_{Q_h} : L^2(\E) \to Q_h$ 
and $\Pi_{V_h} : H(\div) \to V_h$ such that 
\begin{align}
\dx' \Pi_{V_h} v = \Pi_{Q_h} \dx' v \qquad \text{for all } v \in H(\div). \label{eq:commute}
\end{align}
In addition, the following interpolation error estimates hold:
\begin{align*}
\|q - \Pi_{Q_h} q\|_{L^2(\E)} &\le C h^m \|q\|_{H^{m}(\E)}, && 0 \le m \le 1\\
\|v - \Pi_{V_h} v\|_{L^2(\E)} &\le C h^{m+1} \|v\|_{H^{m+1}(\E)}, && 0 \le m \le 1\\
\|v - \Pi_{V_h} v\|_{H(\div)} &\le C h^{m} \|v\|_{H^{m+1}(\E)}, && 0 \le m \le 1. 
\end{align*}
\end{lemma}
\begin{proof}
The interpolation operators are obtained by padding together local operators on every subinterval $T$
which are constructed and analyzed with the usual arguments \cite{BoffiBrezziFortin13}. 
\end{proof}
The \emph{commuting diagram property} \eqref{eq:commute} 
will be important for deriving improved estimates. 
From the local construction of the interpolation operators, 
it is clear that the error estimates can be localized 
which allows to obtain sharper estimates for adapted meshes. 

\subsection{Error estimates}

We now summarize the error estimates for the mixed finite element approximation presented above. 
Taking into account the compatibility condition (A6) and the structure of the approximation spaces, 
we also comment on improved error bounds that do not directly follow from the abstract results.

\medskip 

Let us start with the stationary problem:
We denote by $(\bar u, \bar p)$ and $(\bar u_h,\bar p_h)$ the solution of the system~\eqref{eq:stat1}--\eqref{eq:stat5} and 
its Galerkin approximation stated in Problem~\ref{prob:stath}. 

\begin{theorem}[Error estimate for the stationary problem] $ $\\
Let (A1)--(A4) hold and let $V_h$ and $Q_h$ be chosen as above. 
Then for $0 \le m \le 1$ we have
\begin{align*}
\|\bar u - \bar u_h\|_{H(\div)} + \|\bar p - \bar p_h\|_{L^2} \le C h^{m} \big( \|\bar u\|_{H^{m+1}(\E)} + \|\bar p\|_{H^m(\E)}\big) ,
\end{align*}
provided that $\bar u$ and $\bar p$ are sufficiently smooth. 
The constant $C$ only depends on the network geometry and topology, and on the bounds for the coefficients.
\end{theorem}
\begin{proof}
The estimate follows directly from Theorem~\ref{thm:eestat} and Lemma~\ref{lem:approx}. 
\end{proof}

\begin{remark} \label{rem:erreststat}
Using the condition $\dx' V_h = Q_h$ and the properties of the interpolation operators, 
one can derive the improved estimates
\begin{align*}  
\|\bar u - \bar u_h\|_{L^2} + \|\Pi_{Q_h} \bar p - \bar p_h\|_{L^2} \le C h^{m+1} \|\bar u\|_{H^{m+1}(\E)}
\end{align*}
for $0 \le m \le 1$ and $(\bar u,\bar p)$ sufficiently smooth.
We refer to  \cite[Ch~1]{Boffi08} or \cite[Ch~5]{BoffiBrezziFortin13} for details.
Note that $(\bar u_h,\bar p_h)=\Pi_h(\bar u,\bar p)$, and therefore these estimates also hold for the elliptic projection. 
For smooth solutions, we can thus obtain an error of order $O(h^{2})$. 
\end{remark}

\bigskip 

We now turn to the discretization of the instationary problem:
Let $(u,p)$ denote the solution of \eqref{eq:sys1}--\eqref{eq:sys6} 
and $(u_h,p_h)$ be the one of Problem~\ref{prob:instath}.  
We then have

\begin{theorem}[Error estimate for the instationary problem] $ $\\
Let (A1)--(A4) hold and $V_h$ and $Q_h$ by chosen as above.
Then for $0 \le m \le 1$ and $t \ge 0$
\begin{align*}
&\|u(t) - u_h(t)\|_{L^2} + \|p(t) - p_h(t)\|_{L^2}  \\
&\qquad \le C h^m \big( \|u(t)\|_{H^{m+1}(\E)} + \|p(t)\|_{H^{m}(\E)} \\
& \qquad \qquad \qquad + 
t \sup_{0 \le s \le t} e^{-\gamma (t-s)/2}( \|\dt u(s)\|_{H^{m+1}(\E)} + \|\dt p(s)\|_{H^m(\E)} ) \big),
\end{align*}
provided the solution $(u,p)$ is sufficiently smooth. 
The constant $C$ again only depends on the network geometry and topology, and the bounds for the coefficients.
\end{theorem}
\begin{proof}
The estimate is obtained directly from Theorem~\ref{thm:eeinstat} and Lemma~\ref{lem:approx}. 
\end{proof}

\begin{remark} \label{rem:errestinstat}
Similar as for the stationary problem, one can obtain sharper estimates by employing 
the compatibility condition (A6) and the improved estimates for the elliptic projection
given in Remark~\ref{rem:erreststat}. 
Assume for simplicity that $b \in P_0(T_h(\E))$.
Then 
\begin{align*}
&\|u(t) - u_h(t)\|_{L^2} + \|\Pi_{Q_h} p(t) - p_h(t)\|_{L^2}  \\
&\qquad \le C h^{m+1} \big( \|u(t)\|_{H^{m+1}(\E)} + t \sup_{0 \le s \le t} e^{-\gamma (t-s)/2 } \|\dt u(s)\|_{H^{m+1}(\E)} \big)
\end{align*}
for all $0 \le m \le 1$, provided that the solution $(u,p)$ is sufficiently smooth. 
This result is derived by a careful estimate of the right hand sides $\tilde f(t)$ and $\tilde g(t)$ 
arising in the proof of Theorem~\ref{thm:eeinstat}, and using the improved estimates for the elliptic projection.
For sufficiently smooth solution, the error of the semi-discretization thus is of order $O(h^{2})$.
\end{remark}

\section{Numerical tests} \label{sec:num}

We nowillustrate our theoretical findings with some numerical results. 
As a spatial discretization, we use the mixed finite element approximation with $P_1-P_0$ elements outlined above.
For the time integration, we employ an implicit one-step $\theta$-scheme, which can be shown to yield a uniformly exponentially 
stable full discretization; we refer to \cite{EggerKugler15} for details. 
The time step is chosen so small, such that errors introduced by the time 
discretization can be neglected in all our results.

\subsection{Model problem}

%

For our tests we consider the network displayed in Figure~\ref{fig:network}.

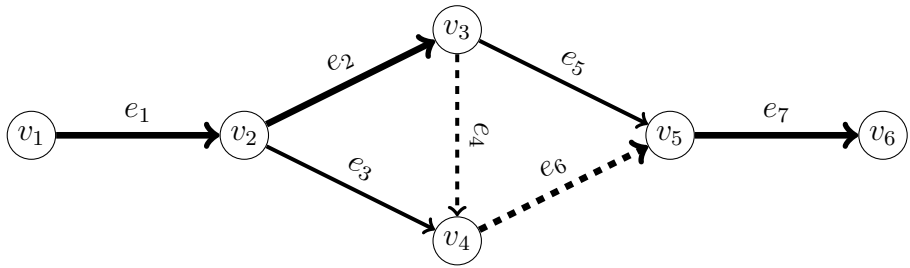
\begin{figure}[ht!]
\begin{center}
  \begin{tikzpicture}[scale=.7]
  \node[circle,draw,inner sep=2pt] (v1) at (0,2) {$v_1$};
  \node[circle,draw,inner sep=2pt] (v2) at (4,2) {$v_2$};
  \node[circle,draw,inner sep=2pt] (v3) at (8,4) {$v_3$};
  \node[circle,draw,inner sep=2pt] (v4) at (8,0) {$v_4$};
  \node[circle,draw,inner sep=2pt] (v5) at (12,2) {$v_5$};
  \node[circle,draw,inner sep=2pt] (v6) at (16,2) {$v_6$};
  \draw[->,thick,line width=2.5pt] (v1) -- node[above,sloped] {$e_1$} ++ (v2);
  \draw[->,thick,line width=2.5pt] (v2) -- node[above,sloped] {$e_2$} ++ (v3);
  \draw[->,thick,line width=1.5pt] (v2) -- node[above,sloped] {$e_3$} ++ (v4);
  \draw[->,thick,line width=1.5pt,style=dashed] (v3) -- node[above,sloped] {$e_4$} ++ (v4);
  \draw[->,thick,line width=1.5pt] (v3) -- node[above,sloped] {$e_5$} ++ (v5);
  \draw[->,thick,line width=2.5pt,style=dashed] (v4) -- node[above,sloped] {$e_6$} ++ (v5);
  \draw[->,thick,line width=2.5pt] (v5) -- node[above,sloped] {$e_7$} ++ (v6);
  \end{tikzpicture}
  \end{center}
 \caption{Network used for numerical tests. 
A spanning tree es obtained by removing the edges marked with dashed lines.
The thickness of the lines corresponds to the diameter of the pipes.\label{fig:network}}
\end{figure}

\noindent
The incidence matrix is given here by 
\begin{align*}
D=
\left(\begin{array}{rrrrrrr}
-1 &  0 &  0 &  0 &  0 &  0 &  0\\
 1 & -1 & -1 &  0 &  0 &  0 &  0\\
 0 &  1 &  0 & -1 & -1 &  0 &  0\\
 0 &  0 &  0 &  0 &  1 &  1 & -1\\
 0 &  0 &  1 &  1 &  0 & -1 &  0\\
 0 &  0 &  0 &  0 &  0 &  0 &  1
\end{array}\right).
\end{align*}
A regular subblock is obtained by removing the first line and the fourth and sixth column,
which amounts to the incidence matrix of the spanning tree with the root vertex removed; 
compare to Remark \ref{rem:graph}.
%
The pipes are chosen to be of unit length, i.e., 
\begin{align*}
 \l=(l_1,\ldots,l_7)=\begin{pmatrix}1& 1& 1& 1& 1& 1& 1\end{pmatrix}.
\end{align*}
The model parameters $a,b,c$ are constant along every pipe with values
\begin{align*}
 &a=\alpha a_0 \qquad \text{with} \qquad a_0=\begin{pmatrix}0.5& 0.5& 4& 4& 4& 0.5& 0.5\end{pmatrix},\\
&b=\begin{pmatrix}4& 4& 1& 1& 1& 4& 4\end{pmatrix} \qquad \text{and} \qquad c=\begin{pmatrix}0.25& 0.25& 1& 1& 1& 0.25& 0.25\end{pmatrix}.
\end{align*}
This amounts to pipes $e_1,e_2,e_6,e_7$ having twice the diameter as the pipes $e_3,e_4,e_5$; see Figure~\ref{fig:network}.
The factor $\alpha$ allows us to adjust the magnitude of the damping in all pipes simultaneously and to investigate 
the dependence of the results on the size of the damping.

\subsection{Estimates for the Poincar\'e constant}

In a first sequence of tests, we investigate the dependence of the constant $C_P$ in
the generalized Poincar\'e inequality 
\begin{align} \label{eq:poincare3}
 \|c^{1/2} u\|_{L^2}^2 \le C_P^2 \big( \|b^{-1/2} \dx' u\|_{L^2}^2 + \|a^{1/2} \Pi^0 u\|_{L^2}^2 \big),
\end{align}
stated in Lemma~\ref{lem:poincare} on the damping factor $\alpha$.
This estimate plays the key role for the decay estimates given in Theorem~\ref{thm:stability} and \ref{thm:stabh} 
and the constant $C_P$ effectively determines the value of the decay rate $\gamma$.
For a single pipe, the Poincar\'e constant $C_P$ can be shown to behave like $C_P^2 \approx \min\{1,1/\alpha\}$; 
compare with \cite[Lemma~A.2]{EggerKugler15}.
We would however expect a similar behaviour also for the simple network considered here.
%
The optimal value for constant $C_P$ in the estimate \eqref{eq:poincare3} is given by the Rayleigh quotient
\begin{align} \label{eq:rayleigh}
C_P^2 = \max_{u \in H(\div)} \frac{\|c^{1/2} u\|^2_{L^2}}{\|b^{-1/2} \dx' u\|_{L^2}^2 + \|a^{1/2} \Pi^0 u\|_{L^2}^2}.
\end{align}
Hence $C_P^2$ ammounts to the largest eigenvalue of the generalized eigenvalue problem 
\begin{align} \label{eq:gevp}
 C u=\lambda (B+A_0) u
\end{align}
with operators $A$, $B$ and $C$ defined by $(A_0 u,v) = (a \Pi^0 u, \Pi^0 v)_\E$, $( B u, v) = (b^{-1} \dx' u, \dx' v)_\E$, and $( C u,v ) = (c u, v)_\E$ for all $u,v \in H(\div)$. 
As before, $\Pi^0 : H(\div) \to H^0(\div)$ 
denotes the projection onto piecewise constant fluxes defined in \eqref{eq:pi0}.

A generalized algebraic eigenvalue problem of similar structure is obtained after discretization.
The largest eigenvalue then corresponds to the discrete Poincar\'e constant
\begin{align} \label{eq:rayleighh}
C_{P,h}^2 = \max_{u_h \in V_h} \frac{\|c^{1/2} u_h\|^2_{L^2}}{\|b^{-1/2} \dx' u_h\|_{L^2}^2 + \|a^{1/2} \Pi^0 u_h\|_{L^2}^2}.
\end{align}
Since we use a conforming discretization $V_h \subset H(\div)$, we clearly get $C_{P_h}^2 \le C_P^2$,   
but by standard estimates for the approximation of elliptic eigenvalue problems \cite{Boffi10}, 
one can expect fast convergence of $C_{P,h}^2$ towards $C_p^2$.
In Table \ref{table:poincareconst} we present the maximal discrete eigenvalues $C_{P,h}^2$ 
for our test problem obtained for different values of the damping parameter $\alpha$ and for a sequence of uniform refinements of the spatial mesh.

\begin{table}[ht!]
\renewcommand{\arraystretch}{1.1}
\begin{center}
\small
\begin{tabular}{l||c|c|c|c|c|c}
$h \setminus \alpha$  
        & $10^{-3}$ & $10^{-2}$ & $10^{-1}$ & $10^{0}$ & $10^{1}$ & $10^{2}$ \\
\hline
\hline
$0.1   $ & $338.53$ & $33.853$ & $3.3853$ & $0.3385$ & $1.0049$ & $1.0049$ \\ 
\hline
$0.05  $ & $338.53$ & $33.853$ & $3.3853$ & $0.3385$ & $1.0111$ & $1.0111$ \\ 
\hline
$0.025 $ & $338.53$ & $33.853$ & $3.3853$ & $0.3385$ & $1.0127$ & $1.0127$ \\
\hline
$0.0125$ & $338.53$ & $33.853$ & $3.3853$ & $0.3385$ & $1.0132$ & $1.0132$ 
\end{tabular}
\medskip
\caption{Optimal discrete  Poincar\'e constants $C_{P,h}^2$ defined by \eqref{eq:rayleighh} 
depending on the value of the damping parameter $\alpha$ and the mesh sizes $h$. }
\label{table:poincareconst}
\end{center}
\end{table}

As expected, the maximal eigenvalues $C_{p,h}^2$ are monotonically increasing when refining the mesh, 
and they converge fast towards the true eigenvalue $C_P^2$ with $h \to 0$.
As for the single pipe, we observe a dependence $C_P^2 \approx \min\{1,1/\alpha\}$ 
on the size of the damping parameter also for the network problem considered here.

\subsection{Exponential stability}

With the next tests, we would like to illustrate the uniform exponential stability and decay of the 
finite element Galerkin approximations discussed in Section~\ref{sec:fem}. 
As initial conditions, we choose $(u_0,p_0) \equiv (0,1)$, 
which corresponds to a solution of the stationary problem \eqref{eq:stat1}--\eqref{eq:stat5} 
with boundary values $p_0(v_1)=p_0(v_6)=1$. 
For the instationary problem, we set the boundary conditions to
\begin{align*}
 p(v_1,t)=p(v_6,t)=\begin{cases}
                    1-t & 0\le t< 1,\\
                    0   & 1\le t.
                   \end{cases}
\end{align*}
According to our theoretical results, the solution should quickly converge towards 
the steady state $(\bar u,\bar p) \equiv (0,0)$. 
In Table~\ref{table:stabh}, we list the values of the discrete energy 
\begin{align*}
\E_h(t):=\frac{1}{2} \Big( \|c^{1/2} u_h(t)\|_{L^2(\E)}^2+\|b^{1/2} p_h(t)\|_{L^2(\E)}^2\Big),
\end{align*}
which corresponds to the approximation of the total energy of the system.
\begin{table}[ht!]
\renewcommand{\arraystretch}{1.1}
\begin{center}
\small 
\begin{tabular}{l||c|c|c|c|c|c||c}
$h \setminus t$                     
         & $0$    & $4$       & $8$       & $12$      & $16$      & $20$       & $\gamma$ \\
\hline
\hline
$0.1000$ & $9.50$ & $1.71507$ & $0.17791$ & $0.01841$ & $0.00190$ & $0.000197$ & $0.540$ \\
\hline
$0.0500$ & $9.50$ & $1.71540$ & $0.17809$ & $0.01844$ & $0.00191$ & $0.000197$ & $0.540$ \\
\hline
$0.0250$ & $9.50$ & $1.71548$ & $0.17813$ & $0.01845$ & $0.00191$ & $0.000198$ & $0.540$ \\
\hline
$0.0125$ & $9.50$ & $1.71550$ & $0.17815$ & $0.01845$ & $0.00191$ & $0.000198$ & $0.540$ 
\end{tabular}
\medskip
\caption{Decay of the discrete energy $\E_h(t)$ for the test problem with parameter $\alpha=1$.
The parameter $\gamma$ is obtained by least-squares fit to the logarithm of the relation $\E_h(t)=C e^{-\gamma t}$ using the data for $t \ge 4$.
\label{table:stabh}}
\end{center}
\end{table}
As can clearly be seen from the results, the decay rate is more or less independent of the meshsize,
which is in perferct agreement with the proofs of Theorem~\ref{thm:stability} and \ref{thm:stabh}.

In a second series of tests, we investigate the dependence of the decay rate $\gamma$ 
on the size of damping parameter. To do so, we repeat the tests on the finest mesh with $h=0.0125$ 
for different values of $\alpha$. 
The corresponding results are displayed in Table~\ref{table:staba}.

\begin{table}[ht!]
\renewcommand{\arraystretch}{1.1}
\begin{center}
\small 
\begin{tabular}{l||c|c|c|c|c|c||c}
$\alpha \setminus t$                     
          & $0$    & $4$       & $8$       & $12$      & $16$      & $20$      & $\gamma$ \\
\hline
\hline
$10^{-3}$ & $9.50$ & $8.09215$ & $8.01978$ & $7.94957$ & $7.88278$ & $7.81723$ & $0.002$ \\
\hline
$10^{-2}$ & $9.50$ & $7.45415$ & $6.81598$ & $6.24595$ & $5.74328$ & $5.28630$ & $0.020$ \\
\hline
$10^{-1}$ & $9.50$ & $3.31009$ & $1.37764$ & $0.59730$ & $0.26706$ & $0.11968$ & $0.197$ \\
\hline
 $10^{0}$ & $9.50$ & $1.71550$ & $0.17815$ & $0.01845$ & $0.00191$ & $0.00020$ & $0.540$ \\
\hline
 $10^{1}$ & $9.50$ & $6.77561$ & $5.46847$ & $4.47603$ & $3.67318$ & $3.01598$ & $0.048$ \\
\hline
 $10^{2}$ & $9.50$ & $8.63295$ & $8.23205$ & $7.93047$ & $7.67813$ & $7.45659$ & $0.009$ 
\end{tabular}
\medskip
\caption{Decay of the discrete energy $\E_h(t)$ for the test problem depending on the parameter $\alpha$.
The decay rate $\gamma$ is obtained by least-squares fit to the logarithm of the relation $\E_h(t)=C e^{-\gamma t}$ using the data for $t \ge 4$.
\label{table:staba}}
\end{center}
\end{table}

By a careful inspection of the proofs of Theorem~\ref{thm:stability} and \ref{thm:stabh}, one would 
expect a behaviour of the decay rate as $\gamma \approx \min\{\alpha,1/\alpha\}$; see \cite{CoxZuazua94,EggerKugler15}
for detailed estimates concerning a single pipe. One would thus expect a reduction in the decay rate for small and large damping parameter $\alpha$, which is exactly what can be observed in our tests. 

\subsection{Error estimates}

Let us finally also study the convergence of the finite element method towards the solution with respect to the meshsize $h$. 
We take the boundary conditions from the previous example and repeat the 
tests for a sequence of uniformly refined meshes and different damping factors $\alpha$.
We use 
\begin{align*}
 e_h=\max_{0 \le t^n \le T} \|u_h^n-u_{2h}^n\|_{L^2}^2+\|p_h^n-p_{2h}^n\|_{L^2}^2
\end{align*}
as a computably measure for the discretization error.
The resulting convergence results are presented in Table \ref{table:errorestimates}.
%
%
\begin{table}[ht!]
\renewcommand{\arraystretch}{1.1}
\small 
\begin{center}
\begin{tabular}{l||c|c|c|c|c|c||c}
$\alpha \setminus h$
          & $0.1 \cdot 2^{-1}$  & $0.1 \cdot 2^{-2}$  & $0.1 \cdot 2^{-3}$  & $0.1 \cdot 2^{-4}$  & $0.1 \cdot 2^{-5}$  & $0.1 \cdot 2^{-6}$  & rate \\
\hline
\hline
$10^{-3}$ & $0.35940$ & $0.05463$ & $0.01541$ & $0.00410$ & $0.00093$ & $0.00020$ & $2.109$ \\ 
\hline
$10^{-2}$ & $0.22003$ & $0.03773$ & $0.00974$ & $0.00257$ & $0.00059$ & $0.00013$ & $2.109$ \\
\hline
$10^{-1}$ & $0.03134$ & $0.00773$ & $0.00192$ & $0.00048$ & $0.00012$ & $0.00003$ & $2.018$ \\
\hline
 $10^{0}$ & $0.02498$ & $0.00611$ & $0.00153$ & $0.00038$ & $0.00010$ & $0.00002$ & $2.006$ \\
\hline
 $10^{1}$ & $0.05493$ & $0.01426$ & $0.00359$ & $0.00090$ & $0.00022$ & $0.00006$ & $1.991$ \\
\hline
 $10^{2}$ & $0.10155$ & $0.03752$ & $0.01062$ & $0.00274$ & $0.00069$ & $0.00017$ & $1.999$
\end{tabular}
\medskip
\caption{
Convergence of the discrete energy error $e_h$ with respect to the mesh size $h$. 
The rates are estimated by least-squares fit to $\log e_h$ for the last two refinement steps.
}
\label{table:errorestimates}
\end{center}
\end{table}
As predicted by the error analysis for the finite element Galerkin method presented in Section~\ref{sec:fem}, 
we can observe second order convergence for the error independent of the size of the damping parameter.

\section{Discussion} \label{sec:disc} 

In this paper, we investigated a linear damped hyperbolic system defined on a one dimensional network. 
Exponential stability and decay estimates could be derived under generic assumptions on 
the network topology and the coefficients of the problem. 
Our analysis relies on a few basic ingredients: an appropriate choice of function spaces;
a variational characterization of solutions; a Poincar\'e type estimate for the network; 
and careful energy estimates. 
The basic steps of our analysis are generic and allow us to analyse very easily also 
the systematic discretization in space by Galerkin methods. 
The analysis can also be extended to time discretization by certain one-step methods.
All important properties of the evolution system derived on the continuous level can be 
preserved on the semi-discrete and fully discrete level.

While our results cover relatively general network topologies and also non-constant coefficients, 
the case of degenerate damping requires different arguments; 
we refer to \cite{BanksItoWang91,ErvedozaZuazua09,Fabiano01} for details concerning the analysis and numerical approximation
in that case.

The main arguments used in our analysis however seem to be appropriate also for other applications; 
examples can be found in \cite{DagerZuazua06,GoettlichHertySchillen15,LagneseLeugeringSchmidt}.
Also the extension to related semi- and quasilinear problems seems feasible without much difficulty 
by the usual perturbation arguments; see e.g. \cite{GattiPata06,Zuazua88} for some results in this direction.

\section*{Acknowledgements}
The authors are grateful for financial support by the German Research Foundation (DFG) via grants IRTG~1529 and TRR~154,
and by the ``Excellence Initiative'' of the German Federal and State Governments via the Graduate School of Computational Engineering GSC~233 at Technische Universität Darmstadt.


\end{document}